\documentclass[a4paper,10pt,leqno]{amsart}
\usepackage[total={6in,9in},
top=1in, left=1in, right=1in, bottom=1in]{geometry}
\usepackage{amsmath}
\usepackage[italian, english]{babel}
\usepackage{amsfonts}
\usepackage{amssymb}
\usepackage{dsfont}
\usepackage{mathrsfs}
\usepackage{graphicx}
\usepackage{setspace}
\usepackage{fancyhdr}
\usepackage{amsthm}
\usepackage{empheq}
\usepackage{cases}
\usepackage[all]{xy}
\usepackage{stmaryrd}
\usepackage{mathrsfs}
\usepackage[colorlinks, citecolor=blue, linkcolor=red]{hyperref}

\newcommand{\varrg}{(M, g)}

\newcommand{\erre}{\mathds{R}}

\newcommand{\enne}{\mathds{N}}

\newcommand{\cinf}{C^{\infty}(M)}

\newcommand{\ricc}{\operatorname{Ric}}
\newcommand{\diver}{\operatorname{div}}
\newcommand{\hess}{\operatorname{Hess}}

\newcommand{\vol}{\operatorname{vol}}

\newcommand{\ra}{\rightarrow}

\newcommand{\set}[1]{{\left\{#1\right\}}}               
\newcommand{\pa}[1]{{\left(#1\right)}}                  
\newcommand{\sq}[1]{{\left[#1\right]}}                  
\newcommand{\abs}[1]{{\left|#1\right|}}                 

\newcommand{\eps}{\varepsilon}                           

\newtheorem{theorem}{\textbf{Theorem}}[section]
\newtheorem{lemma}[theorem]{\textbf{Lemma}}
\newtheorem{proposition}[theorem]{\textbf{Proposition}}
\newtheorem{cor}[theorem]{\textbf{Corollary}}
\newtheorem{defi}[theorem]{\textbf{Definition}}
\theoremstyle{remark}
\newtheorem{rem}[theorem]{\textbf{Remark}}

\numberwithin{equation}{section}

\title[CGSRS]
{Analytic and geometric properties of generic Ricci solitons}

\date{\today} 

\keywords{Ricci solitons, Omori-Yau maximum principle, Rigidity results}

\subjclass[2010]{53C20, 53C25.}

\begin{document}

\maketitle

\begin{center}
\textsc{\textmd{G. Catino\footnote{Politecnico di Milano, Italy.
Email: giovanni.catino@polimi.it. Supported
by GNAMPA projects ``Equazioni differenziali con invarianze in
analisi globale'' and ``Equazioni di evoluzione geometriche e strutture di tipo Einstein''.}, P.
Mastrolia\footnote{Universit\`{a} degli Studi di Milano, Italy.
Email: paolo.mastrolia@gmail.com. Partially supported by FSE,
Regione Lombardia.}, D. D. Monticelli\footnote{Universit\`{a} degli
Studi di Milano, Italy. Email: dario.monticelli@gmail.com. Supported
by GNAMPA projects ``Equazioni differenziali con invarianze in
analisi globale'' and ``Analisi Globale ed Operatori Degeneri''.} and M.
Rigoli\footnote{Universit\`{a} degli Studi di Milano, Italy. Email:
marco.rigoli@unimi.it. \\ The first, the second and the third authors are members of the Gruppo Nazionale per
l'Analisi Matematica, la Probabilit\`{a} e le loro Applicazioni (GNAMPA)
of the Istituto Nazionale di Alta Matematica (INdAM).
}, }}
\end{center}
\begin{abstract}
The aim of this paper is to prove some classification results for generic shrinking Ricci solitons. In particular, we show that every three dimensional generic shrinking Ricci soliton is given by quotients of either $\mathds{S}^3$, $\erre\times\mathds{S}^2$  or $\erre^3$, under some very weak conditions on the vector field $X$ generating the soliton structure. In doing so we introduce analytical tools that could be useful in other settings; for instance we prove that the Omori-Yau maximum principle holds for the $X$-Laplacian on every generic Ricci soliton, without any assumption on $X$.
\end{abstract}

\section{Introduction}\label{sec1}

The fundamental problem of capturing the topological properties of a manifold by its metric structure opened, in the last decades, extremely fruitful areas of
mathematics. From this perspective, there has been an increasing interest in the study of Riemannian manifolds endowed with metrics satisfying special structural
equations, possibly involving curvatures and vector fields. One of the most important example is represented by Ricci solitons, that have become the subject of a rapidly increasing investigation since the appearance of the seminal works of R. Hamilton, \cite{hamilton}, and G. Perelman, \cite{perelman}.
We  recall that if $\varrg$ is a $m$--dimensional, connected,
Riemannian manifold with metric $g$, a \emph{soliton structure}  $(M,
g, X)$ on $M$ is the choice (if any) of a smooth
vector field $X$ on $M$ and a constant $\lambda\in\erre$ such that
\begin{equation}\label{1}
\ricc+\frac{1}{2}\mathcal{L}_Xg =\lambda g,
\end{equation}
where $\ricc$ denotes the Ricci tensor of the metric
$g$ and $\mathcal{L}_Xg$ is the
Lie derivative of the metric in the direction of $X$: the constant
$\lambda$ is sometimes called the soliton constant. The soliton is
 expanding, steady or shrinking if, respectively, $\lambda<0$,
$\lambda=0$ or $\lambda>0$.
If $X$ is the gradient of a potential
$f\in\cinf$ the soliton is called a \emph{gradient Ricci soliton} and
\eqref{1} becomes
\begin{equation}\label{2}
\ricc+\hess(f)=\lambda g.
\end{equation}
In this case, using the symmetry of the tensor
$\hess(f)$, \eqref{2} and the second Bianchi
identity, one proves the validity of the fundamental equation
\begin{equation}\label{3}
\frac{1}{2}\nabla S=\ricc(\nabla f,\,)^\sharp,
\end{equation}
where $S$ is the scalar curvature and ${}^\sharp:T^*M\rightarrow TM$ is
the musical isomorphism. Equation \eqref{3}, together with Hamilton
identity (see e.g. \cite{hamilton})
\begin{equation}\label{1.4Hamilton}
S+\abs{\nabla f}^2-2\lambda f = C, \quad C \in \erre,
\end{equation}
is responsible for a number of basic properties related to the geometry of gradient solitons.
 For instance, from \eqref{1.4Hamilton} one deduces an upper bound on the growth of the weighted volume $\vol_f\pa{B_r} = \int_{B_r}e^{-f}\,d\mu$ of geodesic balls in $M$ (see \cite{caozhou}).
 For generic Ricci solitons $(M, g, X)$, that is,  when $X$ is not the
gradient of a potential, neither \eqref{3} nor Hamilton's identity \eqref{1.4Hamilton} are available. In case of \eqref{3} this is
technically due to the fact that, in the generic setting, the
symmetry of $\hess(f)$, i.e.
$$\hess(f)(Y,Z)=\hess(f)(Z,Y)$$
for every smooth vector fields $Y$ and $Z$,
is replaced by the much more
involved ``commutation rule''
$$g\pa{\nabla_YX,Z}=2\lambda g\pa{
Y,Z}-2\ricc(Y,Z)-g\pa{\nabla_ZX,Y}.$$
Nevertheless, even in in this more general  situation, some important equations valid for gradient solitons still hold, basically in the
same form  (see formulas \eqref{eq3.2} and \eqref{eq3.3} in Lemma \ref{Lem3.1} below).
These equations have been recently considered in \cite{MasRigRim} to infer some upper and lower estimates on $\inf_M S$ and $\sup_M\abs{T}$, where $T$ is the traceless Ricci tensor, under a growth condition on $\abs{X}$. Using Lemma \ref{Lem3.1} below we free $X$ from this restriction, improving on our Theorems 1.1 and 1.2 of \cite{MasRigRim}; for instance, Theorem 1.1 just quoted becomes
\begin{theorem}\label{thmscal}
Let $(M, g, X)$ be a complete, generic Ricci soliton of
dimension $m$ and scalar curvature $S$. Set $S_*=\inf_MS$.
\begin{enumerate}
\item[i)] If $\lambda<0$ then $m\lambda\leq S_*\leq0$. Furthermore, if $S(x_0)=S_*=m\lambda$
for some $x_0\in\ M$, then $(M, g)$ is Einstein and
$X$ is a Killing field; on the other hand, if $S(x_0)=S_*=0$ for
some $x_0\in M$, then $\varrg$ is Ricci flat and $X$ is a homothetic
vector field.
\item[ii)] If $\lambda=0$ then $S_*=0$. Furthermore, if
$S(x_0)=S_*=0$ for some $x_0\in M$, then $\varrg$ is Ricci flat and $X$
is a Killing field.
\item[iii)] If $\lambda>0$ then $0\leq S_*\leq m\lambda$.
Furthermore, if $S(x_0)=S_*=0$ for some $x_0\in M$, then $\varrg$ is
 flat and $X$ is a homethetic vector field, while
$S_*<m\lambda$ unless $M$ is compact, Einstein and $X$ is a Killing
field.
\end{enumerate}
\end{theorem}
Note that, in the original statement of \cite{MasRigRim}, in case iii) we concluded that if $S(x_0)=S_*=0$ then $\varrg$ is Ricci flat; the present stronger conclusion is then obtained by applying Theorem 4.1 of \cite{Tas}. As a matter of fact, part iii) of Theorem \ref{thmscal} and the differential inequality \eqref{eq3.20} below are key steps in the proof of the following classification results for generic shrinking solitons. To fix the notation, we let $o$ be some chosen origin in $M$ and we set $r(x)=\text{dist}_{(M, g)}(x,o)$.

\begin{theorem}\label{thmB}
Let $(M, g, X)$ be a complete generic shrinking Ricci soliton of dimension three. Furthermore, if $M$ is noncompact, assume that the  scalar curvature is bounded and $|\nabla X|=o(|X|)$ as $r \rightarrow \infty$. Then $(M, g)$ is isometric to a  finite quotient of either
$\mathds{S}^3$, $\erre\times\mathds{S}^2$ or $\erre^3$.
\end{theorem}

In higher dimensions Theorem \ref{thmB} generalizes to

\begin{theorem}\label{thmC}
Let $(M, g, X)$ be a complete generic shrinking Ricci soliton of dimension $m>3$. Furthermore, if $M$ is noncompact, assume that the  scalar curvature is bounded and $|\nabla X|=o(|X|)$ as $r \rightarrow \infty$. If, for some $\Lambda>0$,  $|\ricc| \leq \Lambda \, S$  and
\begin{equation}\label{1.4}
|W| \,S \,\leq \, \sqrt{\frac{2(m-1)}{m-2}} \pa{|T|-\frac{1}{\sqrt{m(m-1)}}S}^2 \,,
\end{equation}
then $(M, g)$ is isometric to a  finite quotient of either $\mathds{S}^m$, $\erre\times\mathds{S}^{m-1}$ or $\erre^m$.
\end{theorem}
Theorems \ref{thmB} and \ref{thmC} extend to the non-gradient case the previous results of Perelman \cite{perelman1}, Cao et al. \cite{caochenzhu} and the first author \cite{cat} and provide results in the non conformally flat case, which was treated by the first author et al. \cite{catmantmazz}.

\begin{rem}
  Tracing equation \eqref{1} it follows that the previous theorems in particular hold simply assuming that $\abs{\nabla X}$ is bounded and, if $m>3$, inequality $\eqref{1.4}$.
\end{rem}

In proving Theorems \ref{thmscal}, \ref{thmB} and  \ref{thmC} we shall need
certain geometric and analytic preliminary results that are interesting in their
own, which we state and prove in Sections \ref{sec2} and \ref{sec3}. In particular, in Section \ref{sec2} (see Proposition \ref{Lem3.2}) we show that the Omori-Yau maximum principle, and therefore the weak maximum principle, holds for the $X$-Laplacian (defined below in equation \eqref{eq2.11}) on every generic Ricci soliton, without any assumption on $X$; this extends the previous result in \cite{lopezrio} to the non-gradient setting. In Section \ref{sec3} we give a sufficient condition for the parabolicity of a large class of linear second order operators, including the $X$-Laplacian. Finally, in Section \ref{sec4} we give a classification result for complete generic Ricci solitons with constant scalar curvature (see Theorem \ref{nameless}).

\section{Geometry of a generic Ricci soliton}\label{sec2}

On a generic Ricci soliton structure $(M, g, X)$ we
introduce (see for instance \cite{MasRigRim}) the differential
operator $\Delta_X$, that we call the \emph{$X$--Laplacian} and that acts on a function $u\in \operatorname{Lip}_{loc}(M)$ as
\begin{equation}\label{eq2.11}
\Delta_Xu=\Delta u-g\pa{X,\nabla u}.
\end{equation}
Of course, \eqref{eq2.11} has to be understood in the weak sense.
Note that in case $X=\nabla f$, that is, in case of a gradient Ricci
soliton with potential $f$, the operator $\Delta_X$ coincides with
the symmetric diffusion operator (sometimes called $f$-Laplacian or drifted
Laplacian)
\[
\Delta_fu=\Delta u-g\pa{\nabla f,\nabla
u}=e^f\diver\pa{e^{-f}\nabla u}.
\]
 We recall that the \emph{Omori-Yau maximum principle}\emph{ for $\Delta_{X}$}
holds on $(M, g)$ if,  given any function $u\in C^{2}(M)$ with $u^{*}=\sup_{M}u<+\infty$, there exists a sequence $\left\{z_{k}\right\}_{k}\subset M$ such that
\begin{equation}\label{O-YMaxPr}
\begin{array}{llll}
&\left(i\right)\,u\left(z_{k}\right)>u^{*}-\frac{1}{k},&\left(ii\right)\,\left|\nabla u\left(z_{k}\right)\right|<\frac{1}{k},&\left(iii\right)\,\Delta_{X}u\left(z_{k}\right)<\frac{1}{k},
\end{array}
\end{equation}
for each $k\in\mathds{N}$. We also say that the \emph{weak maximum principle hold for $\Delta_{X}$} if only $\left(i\right)$ and $\left(iii\right)$ in  \eqref{O-YMaxPr} are met.

From now on we shall
freely use the notation of the method of the moving frame referring
to a fixed local orthonormal coframe for computations. We fix the index
convention $1\leq i,j,k,s,\ldots\leq m=\text{dim}M$, we let
$T=\ricc-\frac{S}{m}g$ denote the traceless
Ricci tensor and we set $T_{ij}$, $R_{ijks}$ and $W_{ijks}$ to
denote the components of $T$, of the Riemann and of the Weyl
curvature tensors respectively. Further notation will be clear from the context. The following equations have been obtained
in Lemmas 9 and 10 of \cite{MasRigRim}.

\begin{lemma}\label{Lem3.1}
Let $(M, g, X)$ be a generic Ricci soliton. Then
\begin{eqnarray}
\label{eq3.2}\frac{1}{2}\Delta_XS&=&\lambda
S-|\ricc|^2\,=\,\lambda S-\frac{S^2}{m}-|T|^2,\\
\label{eq3.3}\frac{1}{2}\Delta_X|T|^2&=&|\nabla
T|^2+2\left(\lambda-\frac{m-2}{m(m-1)}S\right)|T|^2+\frac{4}{m-2}\operatorname{tr}(t^3)-2T_{ki}T_{sj}W_{ksij},
\end{eqnarray}
where the endomorphism $t:TM\rightarrow TM$ is defined by
$t(Y)=T(Y,\,)^\sharp$ and $\operatorname{tr}(t^3)$ is the trace of the operator $t^3=t\circ t\circ t$.
\end{lemma}

Having defined $r(x)$ in the Introduction we now deduce an
upper estimate for $\Delta_Xr$ which depends only on a lower bound
for
\begin{equation}\label{2.4}
\ricc_X=\ricc+\frac{1}{2}\mathcal{L}_Xg.
\end{equation}
We denote with $\operatorname{cut}(o)$ the cut-locus of the origin $o$.
\begin{proposition}\label{Prop2.2}
Let $(M, g)$ be a complete Riemannian manifold of
dimension $m\geq 2$ and let $X$ be a vector field on $M$. Suppose that,
for some $F\in C^0(\erre_0^+)$,
\begin{equation}\label{2.5}
\ricc_X\geq -(m-1)F(r)g.
\end{equation}
Then there exist a constant $C>0$ and $\delta>0$ sufficiently small such that
\begin{equation}\label{2.6}
\Delta_Xr(x)\leq C+(m-1)\int_\delta^{r(x)}F(t)\,dt
\end{equation}
pointwise on $M\setminus\left(\{o\}\cup\operatorname{cut}(o)\right)$ and
weakly on $M$.
\end{proposition}

\begin{proof} Fix $x\in
M\setminus\left(\{o\}\cup\operatorname{cut}(o)\right)$ and let
$\gamma:[0,l]\rightarrow M$, with $l=\text{length}(\gamma)$, be a
minimizing geodesic such that $\gamma(0)=o$ and $\gamma(l)=x$. Note
that $F(r(\gamma(t)))=F(t)$ for every $t\in[0,l]$. From B\"{o}chner
formula, see Remark \ref{remark1} below, applied to the distance
function $r$ outside $\{o\}\cup\text{cut}(o)$ we have
$$0=|\hess(r)|^2+\ricc(\nabla r,\nabla r)+\langle\nabla\Delta r,\nabla
r\rangle$$ so that, using the inequality
$$|\hess(r)|^2\geq\frac{(\Delta r)^2}{m-1}$$ it follows that
the function $\psi(t)=(\Delta r)\circ\gamma(t)$, for $t\in(0,l]$
satisfies the Riccati differential inequality
\begin{equation}\label{eq3.3.5}
\psi^\prime+\frac{1}{m-1}\psi^2\leq-\ricc(\dot{\gamma},\dot{\gamma})\qquad\text{on
}(0,l],
\end{equation}
where $\dot{\gamma}(t)$ is the tangent vector of $\gamma$ at time
$t$. We now define
$\psi_X(t)=(\Delta_Xr)\circ\gamma(t)=\psi(t)-\langle X,\nabla
r\rangle\circ\gamma(t)$ so that
$$
\psi^\prime_X=\psi^\prime-\big(\langle X,\nabla r\rangle\circ\gamma\big)^\prime=
\psi^\prime-\frac{1}{2}\mathcal{L}_Xg(\dot{\gamma},\dot{\gamma}).
$$
Thus using \eqref{eq3.3.5} we obtain
$$
\psi^\prime_X\leq-\frac{\psi^2}{m-1}-\ricc(\dot{\gamma},\dot{\gamma})-\frac{1}{2}\mathcal{L}_Xg(\dot{\gamma},\dot{\gamma})
$$
and from \eqref{2.4}
$$
\psi^\prime_X\leq-\ricc_X(\dot{\gamma},\dot{\gamma}).
$$
From
assumption \eqref{2.5} it follows that
$$
\psi^\prime_X(t)\leq(m-1)F(t)
$$
for $t \in (0, l]$. Choosing $\delta>0$ so
small that the geodesic ball $B_\delta(o)$ is inside the domain
of the normal coordinates at $o$ and setting $C=\max_{\partial
B_\delta}\Delta_Xr$, integration for $t$ between $\delta$ and $r(x)$  gives
$$
\Delta_Xr(x)\leq C+(m-1)\int_\delta^{r(x)}F(t)\,dt
$$
pointwise in $M\setminus(\{o\}\cup\text{cut}(o))$, so that
\eqref{2.6} follows. 
To show the validity of this inequality weakly on all of $M$ we argue in a way similar to that of Lemma 2.5 of \cite{PigRigSet}; for the sake of completeness we report the reasoning. By an observation of Cheng and Yau, \cite{chengYau75}, we can consider an exhaustion $\set{\Omega_n}$ of $M\setminus \operatorname{cut}(o)$ by bounded domains with smooth boundaries starshaped with respect to $o$. Fix $n$ and let $\nu$ be the outward unit normal to $\partial\Omega_n$; denote with $\rho(x) = \operatorname{dist}\pa{x, \partial\Omega_n}$ with the convention that $\rho(x)>0$ if $x\in \Omega_n$ and $\rho(x)<0$ if $x\not\in \Omega_n$. Thus $\rho$ is the radial coordinate for the Fermi coordinates relative to $\partial\Omega_n$. By Gauss lemma $\abs{\nabla \rho}=1$ and $\nabla\rho = -\nu$ on $\partial\Omega_n$. Let
\[
\Omega_{n, \eps} = \set{x\in\Omega_n : \rho(x)>\eps}
\]
for some $\eps>0$ sufficiently small and define the Lipschitz function
\[
\psi_\eps(x) = \begin{cases}
  1 &\text{if} \quad x\in\Omega_{n, \eps} \\ \rho(x)/\eps &\text{if} \quad x\in\Omega_n\setminus\Omega_{n, \eps} \\ 0 &\text{if} \quad x\in M\setminus\Omega_{n}.
\end{cases}
\]
Let $\varphi \in C^\infty_0\pa{M}$, $\varphi\geq 0$; then $\varphi\psi_\eps \in W^{1, 2}_0\pa{\Omega_n}$ and $\varphi\psi_\eps\geq 0$. Because of the validity of \eqref{2.6} in $\Omega_n\setminus\set{o}$ and Gauss lemma, having set $G(x)$ for the right-hand side of \eqref{2.6} we have
\begin{align*}
  \int_{\Omega_n}G(x)\varphi\psi_\eps &\geq  \int_{\Omega_n}-g\pa{\nabla r, \nabla\pa{\varphi\psi_\eps}}-g\pa{X, \nabla r}\varphi\psi_\eps  \\ &=-\pa{\int_{\Omega_n}g\pa{\nabla r, \nabla\varphi}\psi_\eps + g\pa{X, \nabla r}\varphi\psi_\eps}-\frac{1}{\eps}\int_{\Omega_n\setminus\Omega_{n, \eps}}g\pa{\nabla r, \nabla \rho}\varphi.
\end{align*}
Therefore, by the co-area formula,
\[
 \int_{\Omega_n}G(x)\varphi\psi_\eps \geq -\pa{\int_{\Omega_n}g\pa{\nabla r, \nabla\varphi}\psi_\eps + g\pa{X, \nabla r}\varphi\psi_\eps}-\frac{1}{\eps}\int_0^\eps dt\int_{\partial\Omega_{n, t}}g\pa{\nabla r, \nabla \rho}\varphi.
\]
Letting $\eps \downarrow 0^+$ we get
\[
 \int_{\Omega_n}G(x)\varphi \geq -\pa{\int_{\Omega_n}g\pa{\nabla r, \nabla\varphi} + g\pa{X, \nabla r}\varphi}+\int_{\partial\Omega_{n}}g\pa{\nabla r, \nu}\varphi,
\]
and since $\Omega_n$ is starshaped,
\[
 \int_{\Omega_n}G(x)\varphi \geq -\pa{\int_{\Omega_n}g\pa{\nabla r, \nabla\varphi} + g\pa{X, \nabla r}\varphi}.
\]
By letting $n\ra +\infty$, observing that $\operatorname{cut}(o)$ has measure $0$ and  $\operatorname{supp}\varphi$ is compact, using Fatou's Lemma the above yields
\[
 \int_{M}G(x)\varphi \geq -\pa{\int_{M}g\pa{\nabla r, \nabla\varphi} + g\pa{X, \nabla r}\varphi},
\]
giving the validity of \eqref{2.6} weakly on all of $M$.
\end{proof}

If we let $(M, g, X)$ be a generic Ricci soliton with $M$
complete, then \eqref{2.5} holds in the form
$$\ricc_X=\lambda g$$ so that we can choose
$F(t)=-\frac{\lambda}{m-1}$. It follows that
\begin{equation*}
\Delta_Xr\leq C-\lambda\pa{r-\delta}
\end{equation*}
on $M\setminus(\{o\}\cup\operatorname{cut}(o))$ and weakly on all of $M$.

Proceeding in a way similar to that in the proof of Theorem~3 in
\cite{AliDajRig} we have the following proposition that improves on
\cite{MasRigRim}.

\begin{proposition}\label{Lem3.2}
Let $(M, g, X)$ be a generic Ricci soliton with
$(M, g)$ complete. Then the Omori--Yau and
therefore the weak maximum principles hold for the operator
$\Delta_X$.
\end{proposition}

We explicitly remark that there are no assumptions on $X$ besides
that of satisfying \eqref{1} for some $\lambda\in\erre$.



With the same reasoning as in the proof of Theorem 1.1 in \cite{MasRigRim}, as a consequence of Lemma \ref{Lem3.1} and Proposition \ref{Lem3.2} we immediately deduce the validity of Theorem \ref{thmscal}. As for Theorem 1.2 of \cite{MasRigRim}, we state here the new improved version; again, its proof follows the same lines of \cite{MasRigRim} with the replacement of Lemma 3.3 there with the improved version in Proposition \ref{Lem3.2}.
\begin{theorem}\label{thm2.8.1}
  Let $\pa{M, g, X}$ be a complete, generic Ricci soliton of dimension $m\geq 3$, scalar curvature $S$, traceless Ricci tensor $T$ and Weyl tensor $W$. Suppose that
  \begin{equation*}
    i)\, S^{*}=\sup_M S(x)<+\infty; \qquad ii)\, \abs{W}^* = \sup_M \abs{W} <+\infty.
  \end{equation*}
  Then, either $\varrg$ is Einstein or $|T|^*=\sup_M |T|$ satisfies
\begin{equation*}
|T|^*\geq\frac{1}{2}\pa{\sqrt{m(m-1)}\lambda-S^{*}\frac{m-2}{\sqrt{m(m-1)}}\,-\sqrt{\frac{m(m-2)}{2}}\abs{W}^*}.
\end{equation*}
In particular, if $\varrg$ is conformally flat, then either $\varrg$ has constant sectional curvature, or
\begin{equation*}
|T|^*\geq\frac{1}{2}\pa{\sqrt{m(m-1)}\lambda-S^{*}\frac{m-2}{\sqrt{m(m-1)}}}.
\end{equation*}
\end{theorem}

We now develop some further auxiliary results.
First we recall the following formula due to B\"{o}chner,
\cite{BocYan}, and rediscovered many times in recent years.

\begin{lemma}\label{Le2.5}
Let $Y$ be a vector field on the Riemannian manifold
$(M, g)$. Then
\begin{equation}\label{eq3.5}
\diver\pa{\mathcal{L}_Yg}(Y)=\frac{1}{2}\Delta|Y|^2-|\nabla
Y|^2+\ricc(Y,Y)+\nabla_Y(\diver Y).
\end{equation}
\end{lemma}

\begin{rem}\label{remark1}
Formula \eqref{eq3.5} is a generalization of the usual B\"{o}chner
formula. To see this let $u\in C^3(M)$ and let $Y=\nabla u$. Then,
since $\frac{1}{2}\mathcal{L}_{\nabla
u}g=\hess(u)$, equation \eqref{eq3.5}
becomes
$$2\diver(\hess(u))(\nabla
u)=\frac{1}{2}\Delta|\nabla
u|^2-|\hess(u)|^2+\ricc(\nabla u,\nabla
u)+g\pa{\nabla\Delta u,\nabla u}.
$$
Since
\[
\diver(\hess(u))(\nabla u)=\ricc(\nabla u,\nabla
u)+g\pa{\nabla\Delta u,\nabla u},
\]
from the above we
immediately deduce that \eqref{eq3.5} is, in this case, equivalent
to
\[
\frac{1}{2}\Delta|\nabla u|^2=|\hess(u)|^2+\ricc(\nabla u,\nabla u)+g\pa{\nabla\Delta u,\nabla u},
\]
that is, the usual B\"{o}chner formula.
\end{rem}

As a consequence of Lemma \ref{Le2.5} we obtain the following (see also \cite{PetWyPac})

\begin{proposition}\label{PR2.7}
Let $(M, g, X)$ be a generic Ricci soliton. Then
\begin{equation}\label{eq3.6}
\frac{1}{2}\Delta|X|^2=|\nabla X|^2-\ricc(X, X),
\end{equation}
or equivalently
\begin{equation}\label{eq3.7}
\frac{1}{2}\Delta_X|X|^2=|\nabla X|^2-\lambda|X|^2.
\end{equation}
\end{proposition}

\begin{proof} We trace the soliton equation \eqref{1} to obtain
$$S+\diver X=m\lambda$$ and from this we deduce
\begin{equation}\label{eq3.8}
\nabla S=-\nabla \diver X.
\end{equation}
On the other hand, contracting twice the second Bianchi's identities
we have the well--known formula
\begin{equation}\label{eq3.9}
\nabla S=2\diver\,\ricc.
\end{equation}
Thus, comparing \eqref{eq3.8} and \eqref{eq3.9}, we obtain
\begin{equation}\label{eq3.10}
\nabla\diver X=-2\diver\,\ricc.
\end{equation}
Now taking the divergence of \eqref{1} and using the fact that
$\diver(\lambda g)=0$ we find
$$\diver(\mathcal{L}_Xg)=-2\diver\,\ricc,$$
and from \eqref{eq3.10} we deduce
$$\nabla\diver X=\diver(\mathcal{L}_Xg).$$
In particular
\begin{equation}\label{eq3.11}
\nabla_X\diver X=\diver(\mathcal{L}_Xg)(X).
\end{equation}
Substituting into \eqref{eq3.5} we immediately obtain \eqref{eq3.6}.
As for \eqref{eq3.7}, using \eqref{eq3.6} and the definition of
$\Delta_X$ given in \eqref{eq2.11}, we have
\[
\frac{1}{2}\Delta_X|X|^2=|\nabla X|^2-\lambda|X|^2+\frac{1}{2}\pa{\mathcal{L}_Xg(X, X)-g\pa{
X,\nabla|X|^2}},
\]
from which \eqref{eq3.7} follows since
$$
\mathcal{L}_Xg(X, X)= g\pa{X,\nabla|X|^2}.
$$
\end{proof}

The following formulas can be verified by a simple direct
computation. For any vector fields $Y, Z$ and functions $u,v\in C^2(M)$, with
$v\neq0$ on $M$,
\begin{equation}\label{2.16}
\Delta_Y\left(\frac{u}{v}\right)=\frac{1}{v}\Delta_Yu-\frac{u}{v^2}\Delta_Yv-2g\pa{\nabla\left(\frac{u}{v}\right),\frac{\nabla
v}{v}},
\end{equation}
\begin{equation}\label{2.16.1}
  \Delta_{Y+Z}u = \Delta_Yu -g\pa{Z, \nabla u}.
\end{equation}
The next computational result uses Proposition \ref{PR2.7}.
\begin{lemma}\label{LE2.8}
Let $\pa{M, g, X}$ be a generic Ricci soliton and
$\alpha\in(0,1]$. Suppose $S>0$ on $M$. Then
\begin{equation}\label{eq3.13}
\Delta_{X-2\nabla\log
S}\left(\frac{|X|^2}{S^\alpha}\right)\leq
\frac{2}{S^\alpha}\bigg\{\pa{\frac{2-\alpha}{\alpha}}|\nabla
X|^2
-\sq{(\alpha+1)\lambda-\alpha\frac{|\ricc|^2}{S}}|X|^2\bigg\}.
\end{equation}
\end{lemma}

\begin{proof}
 Using equations \eqref{eq3.2} and \eqref{eq3.7},
with the aid of \eqref{2.16} and \eqref{2.16.1}, we compute
\begin{align}
\label{eq3.14}
\Delta_{X-2\nabla\log S}\pa{\frac{|X|^2}{S^\alpha}}&=
\frac{2}{S^\alpha}\set{\abs{\nabla X}^2-\left[(\alpha+1)\lambda-\alpha\frac{|\ricc|^2}{S}\right]|X|^2}\\
\nonumber&\quad+\frac{2(1-\alpha)}{S^{\alpha+1}}g\pa{\nabla|X|^2,\nabla
S}+\frac{\alpha(\alpha-1)}{S^{\alpha+2}}|X|^2|\nabla S|^2.
\end{align}

Next we use the inequality $$|\nabla|X|^2|\leq2|X||\nabla X|$$ and
Cauchy and Young's inequalities with $\eps>0$ to
get
\[
\frac{1}{S^{\alpha+1}}g\pa{\nabla S, \nabla|X|^2}\leq \frac{2\abs{X}\abs{\nabla X}\abs{\nabla S}}{S^{\alpha+1}} \leq \frac{1}{\eps}\frac{\abs{\nabla X}^2}{S^\alpha} + \frac{\eps}{S^{\alpha+2}}\abs{X}^2\abs{\nabla S}^2.
\]
Hence, for any $\alpha\in(0,1]$, inserting the previous
inequalities into \eqref{eq3.14} and rearranging terms we obtain
\begin{align*}
\Delta_{X-2\nabla\log S}\pa{\frac{|X|^2}{S^\alpha}}&\leq
\frac{2}{S^\alpha}\set{\sq{1+\frac{1-\alpha}{\eps}}\abs{\nabla X}^2-\left[(\alpha+1)\lambda-\alpha\frac{|\ricc|^2}{S}\right]|X|^2}\\
\nonumber&\quad+\frac{1-\alpha}{S^{\alpha+2}}\pa{2\eps-\alpha}|X|^2|\nabla S|^2.
\end{align*}
Choosing $\varepsilon=\frac{\alpha}{2}$ we finally have
\[
\Delta_{X-2\nabla\log
S}\left(\frac{|X|^2}{S^\alpha}\right)\leq
\frac{2}{S^\alpha}\bigg\{\pa{\frac{2-\alpha}{\alpha}}|\nabla
X|^2
-\sq{(\alpha+1)\lambda-\alpha\frac{|\ricc|^2}{S}}|X|^2\bigg\}.
\]
that is, \eqref{eq3.13}.
\end{proof}
 \vspace{0,4cm}
\begin{cor}\label{CO2.18.1}
  Let $\pa{M, g, X}$ be a generic, shrinking, complete Ricci soliton and assume that $S>0$, $S^*=\sup_M S < +\infty$, $\abs{\ricc}\leq \Lambda S$ for some constant $\Lambda>0$ and
  \begin{equation*}
    \abs{\nabla X} = o\pa{\abs{X}} \quad \text{ as }\, r(x)\ra +\infty.
  \end{equation*}
  Then there exists $\alpha \in (0, 1]$ and a compact $K = K_\alpha \subset M$ such that
  \begin{equation*}
    \Delta_{X-2\nabla\log S}\pa{\frac{\abs{X}^2}{S^\alpha}}<0 \quad \text{ on }\, M\setminus K.
  \end{equation*}
\end{cor}
\begin{proof}
  An immediate consequence of the assumptions and of \eqref{eq3.13}.
\end{proof}

%

Our last ingredient for the proof of the main results comes from Lemma \ref{Lem3.1}.

\begin{lemma}\label{lemma}
Let $(M, g, X)$ be a generic Ricci soliton of dimension $m$
with scalar curvature $S>0$ on $M$. Then
\begin{equation}\label{eq3.20}
\begin{array}{rcl}
\displaystyle\frac{1}{2}\Delta_{X-2\nabla\log
S}\left(\frac{|T|^2}{S^2}\right)&\geq&\displaystyle2\frac{|T|^2}{S^3}\left(|T|-\frac{1}{\sqrt{m(m-1)}}S\right)^2
+\frac{1}{S^3}\left(\frac{|T|}{\sqrt{S}}|\nabla S|-\sqrt{S}|\nabla
T|\right)^2\\
\displaystyle&&\displaystyle-\sqrt{\frac{2(m-2)}{m-1}}\frac{1}{S^2}|W||T|^2.
\end{array}
\end{equation}
\end{lemma}

\begin{proof}
 We use equations \eqref{2.16}, \eqref{eq3.2} and
\eqref{eq3.3} and
\begin{equation*}
|\ricc|^2=|T|^2+\frac{S^2}{m}
\end{equation*}
to compute
\begin{equation}\label{eq3.22}
\Delta_{X-2\nabla\log S}\left(\frac{|T|^2}{S^2}\right)=A+B+C,
\end{equation}
where
\begin{eqnarray*}
A&=&\frac{2}{S^3}\left(S|\nabla T|^2+\frac{1}{S}|T|^2|\nabla S|^2-\langle\nabla|T|^2,\nabla S\rangle\right),\\
B&=&\frac{8}{m-2}\,\frac{1}{S^2}\text{tr}(t^3)+\frac{4}{m(m-1)}\,\frac{|T|^2}{S}+4\frac{|T|^4}{S^3},\\
C&=&-\frac{4}{S^2}T_{ik}T_{sj}W_{ksij}.
\end{eqnarray*}
Next we use Cauchy inequality and $$|\nabla|T|^2|\leq2|\nabla
T||T|$$ to obtain
\begin{equation}\label{eq3.23}
A\geq\frac{2}{S^3}\left(\frac{|T|}{\sqrt{S}}|\nabla
S|-\sqrt{S}|\nabla T|\right)^2.
\end{equation}
Since $T$ is trace free, by Okumura's lemma, \cite{Oku},
\begin{equation*}
\text{tr}(t^3)\geq-\frac{m-2}{\sqrt{m(m-1)}}|T|^3
\end{equation*}
with equality holding if and only if either $|T|=0$ or $|T|=\frac{1}{\sqrt{m(m-1)}}S$. Therefore
\begin{equation}\label{eq3.25}
B\geq4\frac{|T|^2}{S^3}\left(|T|-\frac{1}{\sqrt{m(m-1)}}S\right)^2.
\end{equation}
Finally, by Huisken's inequality, \cite[Lemma 2.4]{Hui},
\begin{equation}\label{eq3.26}
C\geq-\frac{2\sqrt2}{S^2}\sqrt{\frac{m-2}{m-1}}|W||T|^2.
\end{equation}
Inequality \eqref{eq3.20} now follows immediately by putting together
\eqref{eq3.22}, \eqref{eq3.23}, \eqref{eq3.25} and \eqref{eq3.26}.
\end{proof}

\

\section{Some analytic results}\label{sec3}

Let $(M, g)$ be an $m$--dimensional Riemannian
manifold, $X$ a vector field and $T$ a semipositive definite,
symmetric $(0, 2)$-tensor on $M$. We define the operator $L=L_{T, X}$ on $M$
acting on  $u\in Lip_\text{loc}(M)$ by
\begin{equation}\label{eq2.1}
Lu\,=\,L_{T, X}u\,=\,\diver\big(T(\nabla
u,\,)^\sharp\big)-g\pa{X,\nabla u},
\end{equation}
where of course the above has to be understood in the weak sense.
We introduce the following
\begin{defi}\label{def2.3}
i) We say that $M$ is \emph{$L$--parabolic} if each bounded above,
$L$--subharmonic function is constant, that is,  each $u\in Lip_\text{loc}(M)$  with
$u^*=\displaystyle\sup_{M}u<+\infty$ and satisfying
\begin{equation*}
Lu\geq0\quad\text{on }M
\end{equation*}
is constant.

ii) We say that $M$ is \emph{strongly $L$--parabolic} if for each
non--constant $u\in Lip_\text{loc}(M)$  with $u^*<\infty$ and for each $\eta>0$
\begin{equation}\label{eq2.5}
\inf_{\Omega_\eta}Lu<0,
\end{equation}
where
\begin{equation*}
\Omega_\eta=\{x\in M\,:\, u(x)>u^*-\eta\}.
\end{equation*}
\end{defi}

Of course, again, \eqref{eq2.5} has to be interpreted in the weak sense, that is, for some $\eps>0$ there exists $\varphi\in Lip_0\pa{\Omega_\eta}$, $\varphi\geq 0$, $\varphi\not\equiv 0$, such that
\begin{equation*}
  -\pa{\int_{\Omega_\eta}T\pa{\nabla u, \nabla \varphi}+g\pa{X, \nabla u}\varphi} \leq -\int_{\Omega_\eta}\eps\varphi.
\end{equation*}

It is clear that strong parabolicity implies parabolicity, and the converse
can be shown to be true whenever the function obtained as the
maximum between a constant and a $L$--subharmonic function is still
$L$--subharmonic. This is the case for the operators in
\eqref{eq2.1}, as we are going to prove by adapting
an idea of Agmon, \cite{Agm}.

\begin{proposition}\label{prop2.6.1}
Let $L$ be the operator defined in\eqref{eq2.1}; let $u\in Lip_\text{loc}(M)$  satisfy
\begin{equation}\label{eq2.6.2}
Lu\geq0\qquad \text{on }M
\end{equation}
and let $\alpha\in\erre$ be any constant. Then the function defined
by
$$
w(x)=\max\{u(x),\alpha\}
$$
satisfies \eqref{eq2.6.2}.
\end{proposition}

\begin{proof}
 Since $u+\beta$ is still a solution of
\eqref{eq2.6.2} for every $\beta\in\erre$, without loss of
generality we can suppose that $\alpha=0$. In this case $w=u^+$, so that it remains to show that $u^+$ is a
solution of \eqref{eq2.6.2}. Towards this end we recall that \eqref{eq2.6.2}
means
\begin{equation*}
\int_M T(\nabla u,\nabla\varphi)+g\pa{X,\nabla u}\varphi\leq0
\end{equation*}
for every $\varphi\in \text{Lip}_0(M)$ such that $\varphi\geq0$. We
fix such a $\varphi$  and
a constant $\varepsilon>0$. We set
$$u_\varepsilon=\sqrt{u^2+\varepsilon^2},\qquad\varphi_\varepsilon=\frac{1}{2}\left(1+\frac{u}{u_\varepsilon}\right)\varphi$$
and note that $\varphi_\varepsilon$ is still an admissible test
function for \eqref{eq2.6.2}. Furthermore, $u_\varepsilon>|u|$ and
\begin{equation}\label{4}
u_\varepsilon\rightarrow|u|,\qquad\nabla
u_\varepsilon\rightarrow\nabla|u|,\qquad\varphi_\varepsilon\rightarrow\pa{\frac{1+\text{sgn}(u)}{2}}\varphi\qquad\qquad\text{
as }\varepsilon\rightarrow0^+.
\end{equation}
 A
simple computation shows that
$$
T(\nabla u_\varepsilon,\nabla\varphi)=T\left(\nabla
u,\nabla\left(\frac{u}{u_\varepsilon}\varphi\right)\right)-\frac{\varphi}{u_\varepsilon^3}(u_\varepsilon^2-u^2)T(\nabla
u,\nabla u)
$$
so that, since $T$ is, \emph{a fortiori},  semipositive definite, we obtain
$$T(\nabla u_\varepsilon,\nabla\varphi)\leq T\left(\nabla u,\nabla\left(\frac{u}{u_\varepsilon}\varphi\right)\right).$$
From this inequality it immediately follows
\begin{equation}\label{eq2.6.4}
T\left(\nabla\left(\frac{u+u_\varepsilon}{2}\right),\nabla\varphi\right)\leq
T(\nabla u,\nabla\varphi_\varepsilon).
\end{equation}
Now by the definition of subsolution
$$
\int_M T(\nabla u,\nabla\varphi_\varepsilon)+g\pa{X,\nabla
u}\varphi_\varepsilon\leq0
$$
and therefore, using
\eqref{eq2.6.4},
$$
\int_M
T\left(\nabla\left(\frac{u+u_\varepsilon}{2}\right),\nabla\varphi\right)+g\pa{X,\nabla
u}\varphi_\varepsilon\leq0.
$$
Letting
$\varepsilon\downarrow0^+$ we deduce by \eqref{4} and by Fatou's lemma
that
$$
\int_M T(\nabla u^+,\nabla\varphi)+g\pa{X,\nabla
u^+}\varphi\leq0,
$$
 that is, $u^+$ is a subsolution of
\eqref{eq2.6.2}.
\end{proof}

The following is a sufficient condition for the validity of strong $L$-parabolicity.
\begin{theorem}\label{thm3.3}
Let $(M, g)$ be a connected Riemannian manifold and
$L$ be as in \eqref{eq2.1}. Let $\gamma\in Lip_{loc}(M)$ satisfy
\begin{eqnarray}
\label{eq2.7}&&\gamma(x)\rightarrow+\infty\text{ as
}x\rightarrow\infty\\
\label{eq2.8}&&L\gamma<0\text{ outside a compact set.}
\end{eqnarray}
Then $M$ is $L$--parabolic.
\end{theorem}
%

\begin{proof} We reason by contradiction
and we assume the existence of a non--constant $u\in Lip_{loc}(M)$, with
$u^*<+\infty$ and of $\eta>0$ such that
\begin{equation}\label{eq2.9}
Lu\geq0\qquad\text{on }\Omega_\eta.
\end{equation}
First we observe that $u^*$ cannot be attained at any point $x_0\in
M$, for otherwise $x_0\in \Omega_\eta$ and by the strong maximum
principle for the operator $L$ given in Theorem 5.6 of \cite{PucRigSer} we have that $u$ is constantly equal to $u^*$ on the connected component of
$\Omega_\eta$ containing $x_0$. From this and connectedness it follows easily that
$u$ is constant on $M$, a contradiction.

Next for $t\in\erre$ we let
\begin{equation*}
\Lambda_t=\{x\in M\,:\,\gamma(x)>t\}
\end{equation*}
and we define
\begin{equation*}
u^*_t=\sup_{x\in\Lambda_t^c}u(x).
\end{equation*}
Using \eqref{eq2.7} it is not hard to see that $\Lambda^c_t$ is
compact, and therefore $u^*_t$ is attained at some point of
$\Lambda^c_t$. Since $u^*$ is not attained in $M$ and
$\{\Lambda_t^c\}$ is a nested family of compact sets exhausting $M$,
there exists a diverging sequence
$\{t_j\}_{j\in\enne}\subset\erre^+$ such that
\begin{equation}\label{eq2.12}
u^*_{t_j}\nearrow u^*\quad\text{as }j\rightarrow+\infty,
\end{equation}
and we can choose $T_1>0$ sufficiently large so that
\begin{equation*}
u^*_{T_1}>u^*-\frac{\eta}{2}.
\end{equation*}
Without loss of generality we can also suppose to have chosen $T_1$
large enough so that also \eqref{eq2.8} holds on $\Lambda_{T_1}$.
Now fix $\alpha$ satisfying $u^*_{T_1}<\alpha<u^*$. Because of
\eqref{eq2.12} we can find $j\in\enne$ sufficiently large so that
\begin{equation*}
T_2=t_j>T_1\qquad\text{and}\qquad u^*_{T_2}>\alpha.
\end{equation*}
We select $\bar\eta>0$ small enough to satisfy
\begin{equation}\label{eq2.15}
\alpha+\bar\eta<u^*_{T_2}.
\end{equation}
For $\sigma>0$ we define
\begin{equation*}
\gamma_\sigma(x)=\alpha+\sigma\big(\gamma(x)-T_1\big).
\end{equation*}
Then
\begin{equation}\label{eq2.17}
L\gamma_\sigma=\sigma L\gamma<0\qquad\text{on }\Lambda_{T_1}.
\end{equation}
Next we observe that
\begin{equation*}
\alpha\leq\gamma_\sigma(x)\leq\alpha+\sigma(T_2-T_1)\qquad\text{on
}\Lambda_{T_1}\setminus\Lambda_{T_2}
\end{equation*}
and therefore we can choose $\sigma>0$ sufficiently small to guarantee
\begin{equation}\label{eq2.19}
\sigma(T_2-T_1)<\bar\eta
\end{equation}
and then
\begin{equation*}
\alpha\leq\gamma_\sigma(x)<\alpha+\bar\eta\qquad\text{on
}\Lambda_{T_1}\setminus\Lambda_{T_2}.
\end{equation*}
On $\partial\Lambda_{T_1}$ we have
\begin{equation*}
\gamma_\sigma(x)=\alpha>u^*_{T_1}\geq u(x)
\end{equation*}
so that
\begin{equation*}
u(x)-\gamma_\sigma(x)<0\qquad\text{and }\partial\Lambda_{T_1}.
\end{equation*}
Furthermore, if $\bar{x}\in\Lambda_{T_1}\setminus\Lambda_{T_2}$ is
such that $$u(\bar{x})=u^*_{T_2}>\alpha+\bar{\eta}$$ then
$$u(\bar{x})-\gamma_\sigma(\bar{x})\geq u^*_{T_2}-\alpha-\sigma(T_2-T_1)>u^*_{T_2}-\alpha-\bar\eta>0$$
because of \eqref{eq2.15} and \eqref{eq2.19}. Finally \eqref{eq2.7}
and the fact that $u^*<+\infty$ imply
\begin{equation}\label{eq2.23}
(u-\gamma_\sigma)(x)<0\qquad\text{on }\Lambda_{T_3}
\end{equation}
for $T_3>T_2$ sufficiently large. Hence
$$\mu=\sup_{x\in\bar\Lambda_{T_1}}(u-\gamma_\sigma)(x)>0$$
and it is in fact a positive maximum attained at some point $z_0$ in
the compact set $\bar\Lambda_{T_1}\setminus\Lambda_{T_3}$.
Thus
\[
\Sigma = \set{x\in\Lambda_{T_1} : (u-\gamma_\sigma)(x)=\mu} \neq \emptyset.
\]
Furthermore, for $y \in \Sigma$,
\[
u(y)=\gamma_\sigma(y)+\mu>\gamma_\sigma(y)=\alpha+\sigma\pa{\gamma(y)-T_1}>\alpha>u^*_{T_1}>u^*-\frac{\eta}{2},
\]
so that
\[
\Sigma \subset \Omega_\eta,
\]
and by \eqref{eq2.23} $\Sigma \subset \bar\Lambda_{T_1}\setminus \Lambda_{T_3}$, therefore $\Sigma$ is compact. Hence there exists an open neighbourhood of $\Sigma$, $\Sigma_U\subset\Omega_\eta$. Fix $y \in \Sigma$ and $\beta \in \pa{0, \mu}$ and call $\Sigma_{\beta, y}$ the connected component of the set
\[
\set{x \in \Lambda_{T_1} : (u-\gamma_\sigma)(x)>\beta}
\]
containing $y$. We can choose $\beta$ sufficiently close to $\mu$ so that $\overline{\Sigma}_{\beta, y}\subset \Omega_\eta\cap\Lambda_{T_1}$. Note that, since $\beta>0$, $\overline{\Sigma}_{\beta, y}$ is compact. Because of \eqref{eq2.17} and \eqref{eq2.9}
\[
Lu \geq 0 \geq \sigma L\gamma = L\gamma_\sigma \quad \text{ on } \Sigma_{\beta, y}
\]
in the weak sense. Furthermore, $u(x)=\gamma_\sigma(x) +\beta$ on $\partial\Sigma_{\beta, y}$. By Theorem 5.3 of \cite{PucRigSer}, $u(x)\leq\gamma_\sigma(x) +\beta$ on $\overline{\Sigma}_{\beta, y}$. However $y\in \Sigma_{\beta, y}$ and we have
\[
u(y) = \gamma_\sigma(y)+\mu >  \gamma_\sigma(y)+\beta
\]
by our choice of $\beta$. Contradiction.
\end{proof}

We remark that for \eqref{eq2.23} to hold it is enough to require
\begin{equation}\label{eq2.24.1}
u(x)=o\big(\gamma(x)\big)\qquad\text{as }x\rightarrow\infty.
\end{equation}
Hence a careful reading of the above proof yields the validity of
the following Liouville--type result.
\begin{theorem}\label{thm2.24.2}
Let $(M, g)$ be a connected Riemannian manifold and
$L$ as in \eqref{eq2.1}. Let $\gamma\in Lip_{loc}(M)$ satisfy
\eqref{eq2.7} and \eqref{eq2.8}. If $u\in Lip_{loc}(M)$ satisfies
\eqref{eq2.24.1} and $Lu\geq0$ on $M$ then $u$ is constant.
\end{theorem}

\section{Proof of Theorems \ref{thmB} and \ref{thmC}}

In this short section we prove Theorem \ref{thmB} and Theorem \ref{thmC}. First of all, from Theorem \ref{thmscal} we know that the soliton is either  flat or it has positive scalar curvature $S>0$. Moreover, from the growth estimates on the vector field $X$ proved in \cite[Remark 2.2]{naber}, we know that $|X|\rightarrow \infty$ as $r \rightarrow \infty$ (we notice that, under the assumptions of theorems \ref{thmB} and \ref{thmC}, one has that the metric has bounded Ricci curvature). 

In dimension three, every complete shrinking soliton has nonnegative sectional curvature \cite{chen}. Moreover, by Hamilton's strong maximum principle, either $g$ has strictly positive sectional curvature or it splits a line. In this latter case, either the soliton is flat or it is isometric to a quotient of the round cylinder $\mathbb{R}\times \mathbb{S}^{2}$. So from now on, in dimension three, we can assume that the metric has strictly positive sectional curvature. In particular it holds $|\ricc|^{2}<\frac{1}{2} S^{2}$. Moreover, the pinching condition \eqref{1.4} is automatically satisfied, since the W.eyl tensor vanishes in three dimension. Thus, it is sufficient to prove Theorem \ref{thmC}, with $m\geq 3$,  to conclude. Now, the proof follows the arguments in \cite{cat}. Under the assumptions of Theorem \ref{thmC}, Corollary \ref{CO2.18.1} applies. Hence, from Lemma \ref{lemma} and Theorem \ref{thm3.3}, we have that $\frac{|T|^{2}}{S^{2}}$ must be constant on $M$. Moreover, from the proof of Lemma \ref{lemma}, we get that $(M, g)$ is either Einstein or satisfies the identity $|T| = \frac{1}{\sqrt{m(m-1)}}S$. Now, if $m=3$,  this violates the fact that the metric has positive sectional curvature. So $(M, g)$  is Einstein, hence it has constant positive sectional curvature and is a finite quotient of $\mathbb{S}^{3}$. On the other hand, if $m\geq 4$, the pinching assumption \eqref{1.4} on the Weyl curvature implies that $(M, g)$ is either Einstein or has zero Weyl tensor. In the first case, since the metric has positive scalar curvature, we have that $M$ is compact. Moreover, from the pinching condition \eqref{1.4}, we get that
$$
|W|^{2}\,\leq\, \frac{2}{m^{2}(m-1)(m-2)} \,S^{2}\,\leq\, \frac{4}{m(m-1)(m-2)(m+1)}\,S^{2}\,.
$$
Thus, the pinching condition in Huisken \cite[Corallary 2.5]{Hui} is satisfied which implies
that $(M, g)$ has positive curvature. Hence, from a classical theorem of Tachibana \cite{tachib1}, we conclude that $(M, g)$ has constant positive sectional curvature and is a finite quotient of $\mathbb{S}^{m}$. On the other hand, if the Weyl tensor vanishes, from the classification of locally conformally flat shrinking Ricci solitons given in \cite{catmantmazz} we obtain that if $(M, g)$ is non-flat and noncompact, then it must be a finite quotient of $\mathbb{R}\times\mathbb{S}^{m-1}$.

\medskip

This concludes the proof of Theorem \ref{thmB} and Theorem \ref{thmC}.

\

\section{Further remarks}\label{sec4}
In this final section we collect some further observations on generic Ricci solitons. We begin with the following

\begin{lemma}
Let $Y$ be a vector field on $(M, g)$, $S$ the
scalar curvature, $T=\ricc-\frac{S}{m}g$ the
traceless Ricci tensor with corresponding endomorphism
$t:TM\rightarrow TM$ and $l_Yg:TM\rightarrow TM$
the endomorphism corresponding to
$\mathcal{L}_Yg$, that is, for every vector field
$Z$
\begin{equation*}
(l_Yg)(Z)=(\mathcal{L}_Yg(Z))^\sharp.
\end{equation*}
We set $W_Y=T(Y,\,)^\sharp=t(Y)$. Then
\begin{equation}\label{eq3.30}
\diver W_Y=\frac{1}{2}\text{tr}(l_Yg\circ
t)+\frac{m-2}{2m}Y(S).
\end{equation}
\end{lemma}

\begin{proof}
 We give the short proof for completeness. In the
moving frame notation $W_Y$ is the vector field of components
$$W_j=(W_Y)_j=Y_iT_{ij}.$$ Thus
$$\diver W_Y=Y_{ik}T_{ik}+Y_iT_{ik,k}.$$ Using the fact that $T$
is symmetric and Schur's identity $$2R_{ik,i}=S_k$$ we have
\begin{equation*}
\diver W_Y\,=\,\frac{1}{2}(Y_{ik}+Y_{ki})T_{ik}+Y_i(R_{ik,k}-\frac{S_i}{m})\,=\,\frac{1}{2}\text{tr}(l_Yg\circ
t)+\frac{m-2}{2m}S_iY_i,
\end{equation*}
that is, \eqref{eq3.30}.
\end{proof}
\vspace{0,4cm}

As a consequence we obtain

\begin{proposition}\label{prop3.3.1}
Let $(M, g, X)$ be a generic Ricci soliton with scalar
curvature $S$ and traceless Ricci tensor $T$. Then
\begin{equation}\label{eq3.32}
\diver (T(X,\,)^\sharp)=\frac{m-2}{2m}X(S)-|T|^2.
\end{equation}
\end{proposition}

\begin{proof}
 Use the soliton equation \eqref{1}, the definition
of $T$ and $|T|^2=|\ricc|^2-\frac{S^2}{m}$ into \eqref{eq3.30}
to obtain \eqref{eq3.32}.
\end{proof}

\begin{rem}
Note that \eqref{eq3.32} could be interpreted as a kind of
infinitesimal ``Kazdan--Warner'' condition for generic Ricci solitons.
\end{rem}

The following result is an immediate consequence of Proposition
\ref{prop3.3.1}.

\begin{theorem}\label{nameless}
Let $(M, g, X)$ be a complete, generic Ricci soliton
with constant scalar curvature and dimension $m\geq3$. Let $T$ be
the trace free Ricci tensor and assume that for $p,q$ conjugate
exponents
\begin{equation}\label{eq4.2}
|X|\in L^p(M),\qquad|T|\in L^q(M).
\end{equation}
Then $(M, g)$ is Einstein, $X$ is either Killing or
homothetic (but not Killing). In this latter case
$(M, g, X)$ is not steady and
$(M, g)$ is locally Euclidean.
\end{theorem}

\begin{proof}
 Since $S$ is constant, \eqref{eq3.32} becomes
$$
\diver(T(X,\,)^\sharp)=-|T|^2\leq0.$$
Furthermore, because of
\eqref{eq4.2} the vector field $T(X,\,)^\sharp\in L^1(M)$. We apply
Karp's  version of the divergence theorem \cite{karp} to deduce
$\int_M|T|^2=0$, that is, $(M, g)$ is Einstein.
From the soliton equation \eqref{1}
$$\frac{1}{2}\mathcal{L}_Xg=\left(\lambda-\frac{S}{m}\right)g.$$
If $X$ is not Killing $\lambda-\frac{S}{m}\neq0$, and by Theorem 4.1
in \cite{Tas}, $(M, g)$ is locally Euclidean so
that $S\equiv0$ and $\lambda\neq0$.
\end{proof}

\bibliographystyle{plain}

\bibliography{BiblioGeneric_CMMR}

\end{document}